\documentclass[a4paper,11pt]{article}

\usepackage[margin=1in]{geometry}
\usepackage{setspace}
\usepackage{amsmath}
\usepackage{amssymb}
\usepackage{amsthm}
\usepackage{graphicx}
\usepackage{float}
\usepackage{caption}
\usepackage{listings}
\usepackage{tikz}
\usepackage{enumitem}
\usepackage{xcolor}
\usepackage{mathrsfs}

\newtheorem{theorem}{Theorem}[section]
\newtheorem{proposition}[theorem]{Proposition}
\newtheorem{corollary}[theorem]{Corollary}

\newtheorem{lemma}[theorem]{Lemma}

\newtheorem{remark}[theorem]{Remark}

\numberwithin{figure}{section}

\providecommand{\keywords}[1]{\textbf{Keywords:} #1}

\usepackage{hyperref}
\hypersetup{
    pdftitle={On the Rigidity of Analytic Mappings in Complex Analysis and Geometry},
    pdfauthor={Hanwen Liu}
}

\begin{document}

\title{\texorpdfstring{\textbf{On the Rigidity of Analytic Mappings in Complex Analysis and Geometry}}{On the Rigidity of Analytic Mappings in Complex Analysis and Geometry}}
\author{Hanwen Liu}
\date{}

\maketitle

\begin{abstract}
We establish rigidity results for holomorphic mappings and plurisubharmonic functions in complex geometry. 

First, under mild conditions, we show that the gradient of a $\operatorname{U}(1)$-invariant strictly plurisubharmonic function in $\mathbb{C}^2$ possesses finite fibers and induces a analytic mapping of topological degree $1$ on the symplectic quotient. 

Second, we prove that continuous fiber-wise holomorphic maps on proper fibrations elevate to global holomorphic maps when anchored by mutually disjoint sections, yielding rigidity for homomorphisms between elliptic fibrations and Abelian schemes. 

Third, we demonstrate that a fiber-wise holomorphic map of mapping degree $1$ from a fibered compact Kobayashi hyperbolic manifold to a projective variety is a biholomorphism, provided it is injective on a very ample hypersurface. 

Finally, we prove that a holomorphic Lie group action with sufficiently large orbits confines the critical locus of a proper invariant strictly plurisubharmonic function to the fixed-point set, guaranteeing a unique global minimum and yielding a sharp differential topological obstruction on the orbit dimensions of compact Lie group actions.
\end{abstract}

\begin{center}
\end{center}

\tableofcontents
\onehalfspacing
\raggedbottom

\section{Introduction and Background}

The rigid behavior of holomorphic functions and analytic spaces is one of the foundational pillars of complex geometry.
A striking manifestation of this rigidity is the ability to deduce global structures from merely partial, lower-dimensional, or purely topological data.
In 1906, F. Hartogs proved the following fundamental theorem on separate holomorphicity:
\begin{theorem}[F.~Hartogs, \cite{Hartogs}]\label{Hartogs}
    Let $\Delta:=\{z\in\mathbb{C}:|z|<1\}$ be the open unit disc in the complex plane $\mathbb{C}$, and let $f\colon \Delta^n \rightarrow \mathbb{C}$ be an arbitrary function.
    Suppose that for every $(a_1, \dots, a_n) \in \Delta^n$ and every $i \in\{1, \dots, n\}$, the function
    $$g\colon \Delta \rightarrow \mathbb{C}, \quad z\mapsto f(a_1,\dots,a_{i-1},z,a_{i+1},\dots,a_n)$$
    is holomorphic.
    Then $f$ itself is holomorphic.
\end{theorem}

Further results exploring this direction are then established, including Forelli's theorem~\cite{Forelli} for spherical harmonics and Hwang-Mok theorems~\cite{HwangMok1998}\&\cite{HwangMok2005} on Fano varieties. The evolution of these foundational principles has culminated in profound global rigidity results within the setting of higher-dimensional K\"ahler geometry. A central paradigm in this development is Siu’s strong rigidity theorem:

\begin{theorem}[Y-T.~Siu, \cite{Siu}]\label{Siu}
Let $M$ be a compact K\"ahler manifold of complex dimension $n\geq2$ and let $Q$ be a compact quotient of an irreducible bounded symmetric domain of rank greater than $1$. If $M$ is homotopy equivalent to $Q$, then $M$ is biholomorphic or conjugate biholomorphic to $Q$.
\end{theorem}

This result demonstrates that the underlying harmonic maps, constrained by the Bochner-type identity and the curvature of the target, must collapse to global holomorphic isomorphisms. Such phenomena suggest that the presence of negative curvature, or more generally, the absence of rational curves, serves as a geometric anchor that rigidifies the space of mappings.

In accordance with the spirit of rigidity exemplified by Hartogs' theorem and Siu's rigidity theorem, this paper investigates how global geometric and analytic structures can be enforced by partial algebraic or differential constraints.
We divide our exploration into three distinct but conceptually related results concerning the rigidity of mappings.

In Section 2, we explore rigidity within the realm of pluripotential theory and symplectic geometry.
We consider the real analytic gradient mappings of strictly plurisubharmonic functions.
Without assuming global real convexity, the gradient mapping of a strictly pseudoconvex potential is not guaranteed to be injective.
However, we demonstrate that the presence of a global $\operatorname{U}(1)$-symmetry, coupled with a phase condition, restricts the gradient fibers.
We prove that in $\mathbb{C}^2$, the fibers are strictly finite.
We then elevate this local differential property to a global symplectic rigidity theorem, proving that the Marsden-Weinstein symplectic quotient of the moment map perfectly filters out the local pseudoconvex anomalies, yielding a proper mapping of topological degree exactly 1.

In Section 3, we turn our attention to the global extension of fiber-wise holomorphic mappings.
While Hartogs' theorem guarantees the extension of holomorphic functions across coordinates, elevating a continuous, fiber-wise holomorphic mapping to a globally holomorphic one on a proper fibration requires navigating the global geometry of the ambient space.
We demonstrate that by anchoring the mapping to a sequence of holomorphic sections with mutually disjoint images, one can exploit the Noetherian nature of the local rings of the relative morphism space.
This algebraic stabilization forces the mapping to be globally holomorphic, a rigidity result we apply to homomorphisms between elliptic fibrations and Abelian schemes.

In Section 4, we transition to the rigidity of proper holomorphic mappings on hyperbolic spaces.
We present a rigidity theorem demonstrating how continuous, fiber-wise analytic mappings can be automatically elevated to global bi-holomorphic isomorphisms when anchored by a transverse algebraic intersection, specifically a very ample hypersurface, on a Kobayashi hyperbolic manifold.

In Section 5, we present a digression that bridges the pluripotential rigidity from Section 2 with the applied theory of convex optimization. We analyze the optimization landscape of a strictly convex objective function defined on an ambient complex Euclidean space when constrained to a closed complex submanifold. Such a restriction naturally weakens strict convexity to strict plurisubharmonicity. We demonstrate that imposing a holomorphic Lie group action with sufficiently large orbits rigidly forces the critical locus into the fixed-point set, ensuring the existence of a unique global minimum. This optimization theoretic approach unexpectedly provides a sharp geometric obstruction concerning the allowed orbit dimensions for compact Lie group actions on Stein manifolds.

Finally, in Appendix~A, we establish the strict convergence of the associated K\"ahler gradient flow, providing the continuous dynamical foundation necessary for algorithmic implementation.

\section{Pluripotential Theory and Applications to Moment Maps}

In this section, we investigate the geometric rigidity of the gradient mapping associated with a strictly plurisubharmonic function.
If a differentiable function is strictly convex, the theory of monotone operators dictates that its gradient map is a homeomorphism.
However, for strictly plurisubharmonic functions, the lack of convexity permits the level sets to develop folds, locally breaking injectivity.
We establish that imposing a global $\operatorname{U}(1)$-symmetry, combined with the condition that equal gradients imply equal function values, imposes severe dimensional constraints on the gradient fibers.

We first suggest a motivating observation, where strict subharmonicity coincides with the strict positivity of the radial derivative, preserving global injectivity. The following result is standard; however, to fix notations and familiarize the reader with the techniques used later, we provide a detailed proof.

\begin{lemma}\label{d=n}
Let $B$ be the unit ball in $\mathbb{R}^n$, and $u\colon B\rightarrow\mathbb{R}$ a strictly subharmonic $C^2$ function.
Suppose that $u$ is rotation invariant, and for any $x,y\in B$, the equality $\nabla u(x)=\nabla u(y)$ implies $u(x)=u(y)$.
Then, the gradient $\nabla u\colon B\rightarrow\mathbb{R}^n$ is a homeomorphism onto its image.
\end{lemma}
\begin{proof}
Since $u$ is $\operatorname{SO}(n)$-invariant, it is a purely radial function.
We may write $u(x) = f(r)$, where $r = \|x\|$ and $f\colon [0,1) \rightarrow \mathbb{R}$ is $C^2$ continuously differentiable.
The strictly subharmonic condition $\Delta u > 0$ expressed in spherical coordinates yields
$$
\frac{1}{r^{n-1}} \frac{d}{dr} ( f'(r)r^{n-1} ) > 0
$$
for all $r>0$.
Since $u \in C^2(B)$ is rotationally invariant, its differential at the origin vanishes, ensuring $f'(0) = 0$.
Multiplying by $r^{n-1} > 0$ and integrating from $0$ to $r$, we obtain that $ f'(r)r^{n-1} > 0$ for all $r > 0$.
Consequently, the radial derivative satisfies $f'(r) > 0$ for all $r > 0$, which ensures that $u$ is strictly monotonically increasing with respect to the norm $\|x\|$.
This is a direct manifestation of the strong maximum principle, as a strictly subharmonic function cannot achieve a local maximum in the interior of the domain.

To establish injectivity, assume there exist $x,y \in B$ such that $\nabla u(x) = \nabla u(y)$.
By hypothesis, this equality implies $u(x) = u(y)$.
Because $u$ is strictly monotonically increasing with respect to the norm, the equality $f(\|x\|) = f(\|y\|)$ strictly forces $\|x\| = \|y\|$.

Let $r = \|x\| = \|y\|$. If $r = 0$, then $x = y = 0$, and the injectivity holds.
If $r > 0$, the points $x$ and $y$ lie on the same sphere centered at the origin, meaning there exists a matrix $R \in \operatorname{SO}(n)$ such that $y = Rx$.
By the $\operatorname{SO}(n)$-invariance, we have $u(Rx) = u(x)$ for all $R \in \operatorname{SO}(n)$.

Differentiating this identity $u(Rx) = u(x)$ with respect to the spatial variable $x$ yields the gradient equivariance condition
$$
\nabla u(Rx) = R \nabla u(x).
$$
Evaluating this at the specific rotation $R$, we obtain $\nabla u(y) = R \nabla u(x)$.
Substituting the initial assumption $\nabla u(x) = \nabla u(y)$ into this equivariance relation yields
$\nabla u(x) = R \nabla u(x)$,
and thus 
$$ (I - R) \nabla u(x) = 0,$$
where $I$ denotes the identity matrix.
Because $r > 0$ and $f'(r) > 0$, the gradient evaluating to 
$$
\nabla u(x) = f'(r)\frac{x}{r}
$$ 
is non-zero.
Substituting this into the prior equation yields $$(I - R) ( f'(r)x/r ) = 0.$$
Because $f'(r)/r > 0$, we deduce $(I - R)x = 0$, which implies $Rx = x$ and forces $y = x$.
Therefore, the mapping $\nabla u\colon B \rightarrow \mathbb{R}^n$ is injective.
Because $\nabla u$ is a continuous injective mapping from an open subset of $\mathbb{R}^n$ to $\mathbb{R}^n$, Brouwer's invariance of domain theorem guarantees that $\nabla u$ is a homeomorphism onto its image.
\end{proof}

However, in higher dimensions, the gradient map of a strictly plurisubharmonic function may fail to be a homeomorphism.
Nevertheless, by coupling the $\operatorname{U}(1)$-orthogonality with the maximal principles inherent to complex analysis, we can rigidly trap the gradient fibers inside an affine complex plane.
In the setting of $\mathbb{C}^2$, this geometric constraint is sufficient to strictly bound the real dimension of the fibers to zero.

\begin{proposition}\label{d=2}
Let $\Omega$ be a $\operatorname{U}(1)$-invariant domain in $\mathbb{C}^2$ containing the origin, and $u\colon\Omega\rightarrow\mathbb{R}$ a proper strictly plurisubharmonic real analytic function.
Suppose that $u$ is $\operatorname{U}(1)$-invariant, and for any $z,w\in\Omega$ the equality $\nabla u(z)=\nabla u(w)$ implies $u(z)=u(w)$.
Then, the complex gradient $\nabla u\colon\Omega\rightarrow\mathbb{C}^2$ possesses finite fibers.
\end{proposition}
\begin{proof}
Fix a constant vector $c \in \mathbb{C}^2$.
We aim to show that the fiber $M := \{z \in \Omega \mid \nabla u(z) = c\}$ is a finite set.
If $M$ is empty, the result holds trivially. Assume $M$ is non-empty.
Since $u$ is continuous differentiable and proper, any level set of $u$ is compact.
For any $z, w \in M$, the hypothesis $\nabla u(z) = \nabla u(w)$ implies $u(z) = u(w)$.
Thus, the function $u$ evaluates to a constant ${b}\in\mathbb{R}$ on $M$, ensuring $M \subseteq u^{-1}({b})$.
Because $M$ is a closed subset of a compact set, we obtain $M$ itself is compact.
Furthermore, because $u$ is real analytic, the gradient mapping $\nabla u$ is also real analytic.
Consequently, we have that $M$ is a compact real analytic subvariety of $\Omega$.

First, suppose $c = 0$.
For any $z \in \Omega \setminus \{0\}$, the restriction of $u$ to the complex line spanned by $z$ is a purely radial strictly subharmonic function.
The strict subharmonicity forces its radial derivative to be strictly positive, implying $\langle \nabla u(z), z \rangle > 0$.
Thus, $\nabla u(z) \neq 0$ for all $z \neq 0$.
Consequently, if $c = 0$, the fiber $M$ is contained in $\{0\}$ and is trivially finite.
We henceforth assume $c \neq 0$.

If $\dim(M) = 0$, the compactness of $M$ immediately forces it to be a finite set of points.
Assume, for the sake of contradiction, that $\dim(M) \ge 1$.
Then, the regular locus of $M$ contains a smooth 1-dimensional real analytic curve $\Sigma$.
Take any $v \in T_z\Sigma$, and denote by $\langle \cdot, \cdot \rangle$ the Euclidean inner product on $\mathbb{R}^{4}$.
Because $u$ evaluates to the constant $b$ on the entirety of $M$, and $\Sigma \subseteq M$, the scalar function $u$ is constant on $\Sigma$. Consequently, the differential $du$ vanishes on the tangent space $T_z\Sigma$, meaning any tangent vector is orthogonal to the gradient $\nabla u(z) = c$. That is,
\begin{equation}\label{eq:grad_orth}
    \langle c, v \rangle = 0.
\end{equation}
By the $\operatorname{U}(1)$-invariance, we have $u(e^{it}z) = u(z)$ for all $t \in \mathbb{R}$.
Differentiating this identity with respect to $t$ at $t=0$ yields
$$
\langle \nabla u(z), Jz \rangle = 0,
$$
where $J$ denotes the standard complex structure on $\mathbb{C}^2$.
Evaluating on the fiber $M$, we obtain $\langle Jz ,c \rangle = 0$, which is equivalent to $\langle Jc, z \rangle = 0$.
Taking the directional derivative along the tangent vector $v$ of $\Sigma$ then yields
\begin{equation}\label{eq:Jgrad_orth}
    \langle Jc, v \rangle = 0.
\end{equation}

Let $W := \operatorname{span}_{\mathbb{R}}\{ Jc,c\}$.
Because $J$ is a standard complex structure, the real span $W$ is invariant under $J$, forming a $1$-dimensional complex subspace of $\mathbb{C}^2$.
Let $W^\perp$ denote the orthogonal complement of $W$, which is also a $1$-dimensional complex subspace.
Equations (\ref{eq:grad_orth}) and (\ref{eq:Jgrad_orth}) dictate that $c \perp v$ and $Jc \perp v$, meaning $v \in W^\perp$.
Since the tangent vector of $\Sigma$ is globally confined to the fixed complex line $W^\perp$, the entire connected real curve $\Sigma$ is contained within a single affine complex plane ${L}$ parallel to $W^\perp$.
We now examine the restriction of $u$ to this affine complex plane ${L}$.
Define $M_L := M \cap L$. Because $M$ is a compact real analytic subvariety of $\Omega$, the closed intersection $M_L$ is a compact real analytic subvariety of the $2$-dimensional real affine plane $L$, and it contains $\Sigma$. 

Let $\varphi$ be the restriction of $u$ on $L$. Because $u$ is strictly plurisubharmonic on $\Omega$, its restriction $\varphi$ is a strictly subharmonic function on the intersection ${L} \cap \Omega$.
Because $M_L$ is a compact real analytic variety, Sullivan's Theorem~\cite{Sullivan} dictates that the local Euler characteristic of the link of any point in $M_L$ is even.
This ensures $M_L$ decomposes into Eulerian graphs, thereby containing a closed cycle $C$.
Since $C$ is a subset of $M_L \subseteq M$, the function evaluates to $\varphi = {b}$ on the entirety of $C$.
Let $U$ denote the bounded domain enclosed by the cycle $C$ within the affine plane $L \cong \mathbb{C}$.
Because $u$ is a proper strictly plurisubharmonic function on $\Omega$, it constitutes an exhaustion function, ensuring that $\Omega$ is a Stein manifold.
By the pseudoconvexity of Stein manifolds, the holomorphically convex hull of the compact set $C$ with respect to $\Omega$ is compact.
Since $C$ bounds the domain $U$ in the affine plane $L$, the maximum modulus principle dictates that $U$ is entirely contained within this holomorphically convex hull.
Consequently, the closure of $U$ is a compact subset of $\Omega$, guaranteeing the strict inclusion $U \subseteq \Omega$.
By the strong maximum principle for subharmonic functions, the restricted subharmonic function $\varphi$ cannot attain an interior maximum.
Therefore, on the bounded domain $U$, we must have $\varphi|_U < {b}$.
Consequently, $\varphi$ attains its maximum on the boundary $C$. By the Hopf Lemma~\cite{GilbargTrudinger}, the outward normal derivative of $\varphi$ at any smooth point on $C$ must be positive, namely,
$$
\frac{\partial \varphi}{\partial n} > 0
$$
almost everywhere.
However, the full gradient of $u$ at any point $z\in C$ is $\nabla u (z)= c$.
Because ${L}$ is parallel to $W^\perp$, and $c \in W$, the vector $c$ is orthogonal to the plane ${L}$.
Therefore, the projection of the gradient onto ${L}$ is identically zero.
This implies the normal derivative $\partial \varphi/\partial n$ is exactly $0$ everywhere on $C$, which contradicts the Hopf Lemma.
Therefore, we have $\dim(M) \le 0$. The proof is completed.
\end{proof}

\begin{remark}
The finite fiber property established in Propersition~\ref{d=2} is a low-dimensional phenomenon strictly confined to $2$-dimensional Kähler geometry. For general $d \ge 3$, the dimension counting in the proof inherently breaks down. Specifically, the orthogonal complement the complex line spanned by the gradient possesses complex dimension $d-1\geq2$ when $d \ge 3$, which can no longer trigger the Hopf Lemma contradiction.
\end{remark}

We now elevate this local finiteness property to a global topological rigidity result using the framework of symplectic geometry.
The strict plurisubharmonicity of the potential $\phi$ induces a Kähler metric, and the global $\operatorname{U}(1)$-action is generated by a real-valued Hamiltonian function known as the moment map.
By performing a Marsden-Weinstein symplectic reduction \cite{MarsdenWeinstein} over the level sets of the moment map, we translate the gradient mapping into an induced map between Riemann spheres.

\begin{corollary}\label{Moment}
Let $\phi\colon\mathbb{C}^2\rightarrow\mathbb{R}$ be a uniformly strictly plurisubharmonic real analytic function, and $\mu$ the moment map of $dd^c\phi$.
Suppose that $\phi$ is $\operatorname{U}(1)$-invariant, and $dd^c\phi \ge \varepsilon \omega$ for some $\varepsilon > 0$, where $\omega$ is the standard Kähler form of $\mathbb{C}^2$.
Then, for any regular value $b\in\mu(\mathbb{C}^2)$, the real analytic mapping
$$
f\colon\mu^{-1}(b)/\operatorname{U}(1)=\mathbb{P}^1\rightarrow\mathbb{P}^1,\quad a\mapsto[\partial_{\bar{z}}\phi(a):\partial_{\bar{w}}\phi(a)]
$$
is of topological degree $1$.
\end{corollary}
\begin{proof}
Let ${\Omega} := \mathbb{C}^2 \setminus \{0\}$. Because $\phi$ is uniformly strictly plurisubharmonic, the moment map $\mu \colon \mathbb{C}^2 \rightarrow \mathbb{R}$ grows at least quadratically, ensuring it is a proper real analytic function.
By Sard's theorem for smooth manifolds, the set of critical values of $\mu$ has Lebesgue measure zero.
Fix a regular value $b \in \mu({\Omega})$. 

Because $b$ is a regular value and $\mu$ is proper, the level set $\mu^{-1}(b)$ is a compact, smooth real 3-dimensional manifold.
By the strict plurisubharmonicity of $\phi$, this level set is diffeomorphic to the 3-sphere $\mathbb{S}^3$.
The $\operatorname{U}(1)$-action on $\mathbb{C}^2$ restricts to a free action on $\mu^{-1}(b)$, and the symplectic quotient $X_b := \mu^{-1}(b) / \operatorname{U}(1)$ is a compact real analytic 2-manifold bi-holomorphic to the complex projective line $\mathbb{P}^1$.

Let $\nabla \phi(a) := 2(\partial_{\bar{z}}\phi(a), \partial_{\bar{w}}\phi(a))$ denote the complex gradient vector.
The global $\operatorname{U}(1)$-invariance of $\phi$ dictates the equivariance relation $\nabla \phi(e^{it}a) = e^{it}\nabla \phi(a)$ for all $t \in \mathbb{R}$.
Consequently, the complex line spanned by the gradient is invariant under the phase action, yielding that $\nabla \phi(e^{it}a)$ and $\nabla \phi(a)$ represent the same homogeneous coordinates in $\mathbb{P}^1$.
Therefore, the induced mapping $f \colon X_b \rightarrow \mathbb{P}^1$ is a well-defined real analytic mapping.
We first compute the topological mapping degree of $f$. Restrict the function $\phi$ to the 1-dimensional complex line $L_a := \{\zeta a \mid \zeta \in \mathbb{C}\}$ passing through the origin and a point $a \in \mu^{-1}(b)$.
Define $v(\zeta) := \phi(\zeta a)$. Because $\phi$ is strictly plurisubharmonic and $\operatorname{U}(1)$-invariant, $v$ is a radial and strictly subharmonic function on $\mathbb{C}$.
The strict subharmonicity condition $\Delta v > 0$ strictly forces the radial derivative to be positive for all non-zero radii, which yields the inequality
$$
\langle \nabla \phi(a), a \rangle > 0
$$
for all $a\in\Omega$, where $\langle \cdot, \cdot \rangle$ denotes the real Euclidean inner product on $\mathbb{R}^4$.
Define a continuous homotopy $H_s \colon \mu^{-1}(b) \rightarrow \mathbb{C}^2$ by $H_s(a) := (1-s)\nabla \phi(a) + s a$ for $s \in [0,1]$.
We evaluate its projection onto the outward radial vector $a$ as
$$
\langle H_s(a), a \rangle = (1-s)\langle \nabla \phi(a), a \rangle + s|a|^2.
$$
Since both $\langle \nabla \phi(a), a \rangle > 0$ and $|a|^2 > 0$, their convex combination is strictly positive for all $s \in [0,1]$.
This guarantees that $H_s(a) \neq 0$ anywhere on $\mu^{-1}(b)$. Because $H_s$ maintains the $\operatorname{U}(1)$-equivariance, it descends to a well-defined homotopy $h_s([a]) := [H_s(a)]$ on the quotient space $X_b$.
The homotopy continuously deforms our mapping $h_0 = f$ into the identity mapping $h_1([a]) = [a]$ on $\mathbb{P}^1$.
Since the topological degree is invariant under continuous homotopies, we obtain $\deg(f) = 1$. This completes the proof.
\end{proof}

\begin{remark}
The significance of Corollary~\ref{Moment} lies in its ability to extract a highly rigid topological structure out of a real analytic PDE condition.
The lack of real convexity permits the gradient map in $\mathbb{C}^2$ to fold, generating multiple preimages.
However, by transitioning to the symplectic quotient, we filter out this local geometric noise.
The resulting induced map on $\mathbb{P}^1$ is a proper mapping of topological degree exactly 1, proving that macroscopically, the gradient mapping fundamentally preserves the complex projective geometry of the target space without extraneous winding.
\end{remark}

\section{Regularity of Fiber-wise Holomorphic Mappings}

We transition to investigating the conditions under which a continuous mapping, known to be holomorphic on the individual fibers of a proper fibration, can be elevated to a global holomorphic mapping.
The strategy relies on translating the pointwise agreement of sections into an infinite intersection of complex analytic subsets within the relative morphism space.
By invoking the Noetherian property of analytic local rings, this descending chain must stabilize, enforcing global holomorphicity.
\begin{lemma}\label{Main_Lemma}
Let $X,Y,Z$ be complex manifolds with $\dim(Y)=\dim(X)+1$ and $\phi\colon Y\rightarrow X$ a surjective proper holomorphic map with connected fibers.
Let $\{s_i\}_{i\geq1}$ be a sequence of holomorphic sections of $\phi$ with mutually disjoint images, and let $f\colon Y\rightarrow Z$ be a continuous mapping.
Assume that $f|_{\phi^{-1}(x)}$ and $f\circ s_i$ are holomorphic for every $x\in X$ and all $i\geq1$. Then $f$ is holomorphic.
\end{lemma}
\begin{proof}
Let $U$ be a Zariski dense open subset of $X$ consisting of regular values of $\phi$.
For any $x\in U$, the fiber $Y_x:=\phi^{-1}(x)$ is a compact Riemann surface.
Since $Y_x$ is compact, the infinite set of points $\{s_i(x)\}_{i\geq1}$ possesses an accumulation point in $Y_x$.
By the analytic continuation principle for complex manifolds, any two holomorphic mappings from $Y_x$ to $Z$ that coincide on $\{s_i(x)\}_{i\geq1}$ are identical.
Consider the relative morphism space $$
\pi\colon\operatorname{Hom}_U(\phi^{-1}(U),Z\times U):=\{(g,x):g\in\operatorname{Hom}(Y_x,Z),x\in U\}\rightarrow U,\quad(g,x)\mapsto x 
$$ over $U$, of which connected components are complex analytic spaces by Douady's theorem \cite{Douady}.
Let $\sigma\colon U\rightarrow\operatorname{Hom}_U(\phi^{-1}(U),Z\times U)$ be the continuous section defined by $\sigma(x):=(f|_{\phi^{-1}(x)},x)$, and denote by $\mathcal{H}$ the connected component of $\operatorname{Hom}_U(\phi^{-1}(U),Z\times U)$ containing $\sigma(U)$.
For each integer $k\geq1$, we define a holomorphic mapping $\Phi_k\colon\mathcal{H}\rightarrow (Z\times\cdots\times Z) \times U$ by $$\Phi_k(g,x):= (g(s_1(x)),\dots,g(s_k(x)), x),$$and define$$A_k := \Phi_k^{-1}(\{(f(s_1(x)), \dots, f(s_k(x)), x):x\in U\}).$$ Then $\{A_k\}_{k\geq1}$ is a descending chain of complex analytic subsets in $\mathcal{H}$.
By the analytic continuation principle, it holds that $$\bigcap_{k=1}^\infty A_k=\sigma(U).$$

Because the local rings of $\mathcal{H}$ are Noetherian, the descending chain $\{A_k\}_{k\geq1}$ must stabilize locally in a neighborhood of any point in $\sigma(U)$.
Consequently, the restriction $\pi|_{\sigma(U)}$ is a bijective holomorphic mapping onto its image.
Since $U$ is smooth, Zariski's main theorem \cite{GunningRossi} immediately yields that $\sigma$ is holomorphic.
By the universal property of the relative morphism space, the mapping $f|_{\phi^{-1}(U)}$ itself is holomorphic.
Since $f$ is continuous and the singular locus $Y-\phi^{-1}(U)$ is a proper analytic subset of $Y$, Riemann's extension theorem applied locally in the coordinate charts of $Z$ guarantees that $f$ is holomorphic on the entirety of $Y$.
\end{proof}

The application of Lemma \ref{Main_Lemma} converts algebraic data, namely the stabilization of descending analytic chains, into global analytic rigidity.
We now apply this machinery to elevate continuous homomorphisms between elliptic fibrations and Abelian schemes into holomorphic morphisms.
By utilizing the group law of the elliptic fibration, a single torsion-free holomorphic section generates the required sequence of disjoint sections.
\begin{corollary}\label{Abel}
Let $\pi\colon E\rightarrow B$ be a smooth elliptic fibration and $\varphi\colon A\rightarrow B$ an Abelian scheme over a complex manifold $B$.
Let $f\colon E\rightarrow A$ be a continuous mapping such that $\pi=\varphi\circ f$ and
$$
f_b\colon \pi^{-1}(b)\rightarrow \varphi^{-1}(b), \quad P\mapsto f(P)
$$
is a homomorphism for every $b\in B$.
Assume there exists a holomorphic section $\sigma$ of $\pi$ such that $f\circ\sigma$ is holomorphic and $\sigma(b)$ is torsion-free for every $b\in B$.
Then $f$ is holomorphic.
\end{corollary}
\begin{proof}
For each integer $k\geq1$, we define a section $s_k\colon B\rightarrow E$ by $s_k(b) := k\sigma(b)$, where the multiplication is taken with respect to the group structure of the elliptic fibration $\pi$.
Since $\sigma(b)$ is torsion-free for every $b\in B$, the sequence of sections $\{s_k\}_{k\geq1}$ has mutually disjoint images.
Since $f$ is continuous, we have that $f_b$ is also continuous for all $b\in B$.
But any continuous homomorphism between complex Lie groups is holomorphic, ensuring that $f|_{\pi^{-1}(b)}$ is holomorphic for all $b\in B$.
Because the restriction $f_b\colon \pi^{-1}(b)\rightarrow \varphi^{-1}(b)$ is a group homomorphism for every $b\in B$, we obtain
$$
(f\circ s_k)(b) = f_b(k\sigma(b)) = k f_b(\sigma(b)) = k(f\circ\sigma)(b)
$$
for all $b\in B$.
Therefore we have the identity $f\circ s_k = \mu_k \circ (f\circ\sigma)$, where $\mu_k\colon A\rightarrow A$ denotes the integer multiplication morphism of the Abelian scheme $A$.
Since $A$ is Abelian, the mapping $\mu_k$ is holomorphic. By the assumption that $f\circ\sigma$ is holomorphic, the composition $f\circ s_k$ is holomorphic for every integer $k\geq1$.
The desired result then follows immediately from Lemma~\ref{Main_Lemma}.
\end{proof}

\begin{remark}
The proof of Corollary~\ref{Abel} encapsulates the overarching theme of this section: rigid global structures in complex geometry can be rigorously enforced by combining proper fibrations with the positive constraint of a torsion-free section.
The algebraic stabilization within the relative morphism space serves as the essential bridge between continuous topological data and complex analytic rigidity.
\end{remark}

\section{Rigidity of Mappings from Kobayashi Hyperbolic Manifolds}

We conclude with a different flavor of geometric rigidity.
In this section, we demonstrate that continuous maps from a Kobayashi hyperbolic manifold can be elevated to global isomorphisms, using a very ample hypersurface to lock the complex moduli.
\begin{theorem}\label{Kobayashi}
Let $\phi\colon X\rightarrow Y$ be a continuous map of mapping degree $1$ between compact complex $n$-manifolds of equal dimension $n\geq4$, where $Y$ is projective and $X$ is Kobayashi hyperbolic.
Suppose that there exist a holomorphic fibration $f\colon X\rightarrow\mathbb{P}^1$ and a very ample hypersurface $H$ in $X$ such that $\phi|_{f^{-1}(t)}$ is holomorphic for every $t\in\mathbb{P}^1$ and $\phi|_H$ is injective.
Then $\phi$ is a bi-holomorphic isomorphism. 
\end{theorem}
\begin{proof}
We divide the proof into three main steps: Establishing that the restrictions of $\phi$ on the fibers are finite, then proving their images are mutually disjoint, and finally upgrading the continuous map $\phi$ to a biholomorphism.

\textbf{Step~1:} Since $H\in\operatorname{div}(X)$ is ample, the complex manifold $X$ is projective by Kodaira's embedding theorem \cite{GriffithsHarris}.
Let $$\Sigma:=\mathbb{P}^1-\{b_1,\dots,b_m\}$$ be the punctured Riemann sphere consisting of the regular values of $f$.
For any $t \in \Sigma$, the fiber $X^t := f^{-1}(t)$ is a smooth projective subvariety of $X$ of pure dimension $n-1$.
Since $\phi^t := \phi|_{X^t}$ is holomorphic and $Y$ is projective, Remmert's proper mapping \cite{GunningRossi} theorem and Chow's theorem \cite{GriffithsHarris} ensure that the image $Y^t := \phi(X^t)$ is a projective subvariety of $Y$.
We first establish the injectivity of the restriction $\phi^t$ for any fixed $t \in \Sigma$.
For the sake of contradiction, we now assume that $\phi^t$ is not injective. 
Suppose first that $\dim(Y^t)\leq n-3$.
By Chevalley's fiber dimension theorem, the generic fiber of $\phi^t$ would have dimension at least 2, and hence would contain a projective surface.
Since the intersection of the ample divisor $H$ with a projective surface contains a curve in $H$, the morphism $\phi^t$ would contract a curve in $H$ to a point, violating the injectivity of $\phi|_H$.
Therefore, we must have at least $\dim(Y^t) \ge n-2$. 

Suppose now that $\dim (Y^t) = n-2$.
Take a generic point $y\in Y^t$. Then the fiber $C$ of $\phi^t$ over $y$ is a projective curve.
Since $H \in \operatorname{div}(X)$ is a very ample divisor, $H^t:=X^t\cap H$ is very ample in $X^t$, making $H^t\cap C$ non-empty and $\phi(H^t)=Y^t$.
Because $\phi|_H$ is injective by assumption, the mapping degree of $\phi^t|_{H^t}$ onto its image is exactly 1.
By choosing $y$ to be generic, we ensure that the restriction $\phi^t|_{H^t}$ is unramified at its unique preimage $x \in H^t$, rendering its differential an isomorphism at $x$.
As a result, the tangent space $T_x H^t$ intersects trivially with the kernel of the differential of $\phi^t$ at $x$, which is precisely $T_x C$.
This proves that $C$ intersects $H^t$ transversally, and hence their intersection product is $H^t \cdot C = 1$.
Because $H$ is a very ample divisor, intersecting a projective curve $C$ at a single point with a multiplicity of exactly 1 forces $C$ to be rational, that is, bi-holomorphic to $\mathbb{P}^1$.
However, as $X$ is Kobayashi hyperbolic, there exists no non-constant holomorphic mapping from $\mathbb{P}^1$ into $X$ by Brody's theorem \cite{Brody} on hyperbolicity.
This is a contradiction. Therefore, we have in fact $\dim(Y^t) = n-1$, so the regular mapping $\phi^t\colon X^t\rightarrow Y$ sends its domain of definition onto a projective variety $Y^t$ of equal dimension $n-1$.
We claim that $\phi^t$ is actually a finite morphism. Consider the Stein factorization $\phi^t = \phi^t_1 \circ \phi^t_2$ of $\phi^t$, where $\phi^t_2\colon X^t \rightarrow N$ is a surjective projective morphism with connected fibers onto a normal projective variety $N$, and $\phi^t_1\colon N \rightarrow Y^t$ is a finite morphism.
Since $\phi^t_1$ is finite and surjective, we obtain $$\dim(N) = \dim(Y^t) = n-1.$$ 
Because $\phi^t_2$ is a surjective morphism with connected fibers between projective varieties of the same dimension, its generic fiber is a single point, implying that $\phi^t_2$ is a birational morphism.
Assume to the contrary that $\phi^t$ admits a fiber $F$ of positive dimension.
Since $\phi^t_1$ is finite, this fiber $F$ must be contracted to a point entirely by $\phi^t_2$.
Therefore, the exceptional locus of the birational morphism $\phi^t_2$ is non-empty.
Now, recall from birational geometry that the exceptional locus of a proper birational morphism from a smooth projective variety is uniruled, meaning it is covered by rational curves.
However, as $X$ is Kobayashi hyperbolic, by Brody's theorem \cite{Brody}, it contains no rational curves, which leads to a contradiction.
This proves that $\phi^t$ is indeed a finite morphism. 

\textbf{Step~2:} Now, we prove that $Y^{t_1}$ and $Y^{t_2}$ cannot intersect for distinct $t_1,t_2\in\Sigma$.
Assume to the contrary that there exist points $x_1 \in X^{t_1}$ and $x_2 \in X^{t_2}$ for distinct $t_1,t_2\in\Sigma$ such that $\phi(x_1) = \phi(x_2)$.
Since $\phi^t$ is a finite surjective morphism onto its image for every $t \in \Sigma$, the images $Y^{t_1}$ and $Y^{t_2}$ are projective subvarieties of $Y$ of dimension $n-1$.
By Krull's dimension theorem, the intersection $Y^{t_1} \cap Y^{t_2}$ is a projective subvariety of $Y$ with dimension bounded below by 
\[ \dim (Y^{t_1} \cap Y^{t_2}) \ge (n - 1) + (n - 1) - n = n-2 \ge 2. \]

Consider the incidence variety 
$$I := \{(a_1, a_2) \in X^{t_1} \times X^{t_2} \mid \phi(a_1) = \phi(a_2)\}.$$
Because $\phi^{t_1}$ and $\phi^{t_2}$ are finite morphisms, the natural projection from $I$ onto $Y^{t_1} \cap Y^{t_2}$ has finite fibers.
Therefore, $I$ is a projective subvariety of the fiber product $X^{t_1} \times X^{t_2}$ with $$\dim (I) = \dim(Y^{t_1} \cap Y^{t_2}) \ge 2.$$ 
Let $p_1 \colon I \rightarrow X^{t_1}$ and $p_2 \colon I \rightarrow X^{t_2}$ be the restrictions of the canonical projections onto each factor.
Because $\phi^{t_1}$ and $\phi^{t_2}$ are finite morphisms, the projections $p_1$ and $p_2$ are also finite.
Therefore, the pullbacks $p_1^*(H)$ and $p_2^*(H)$ are both ample divisors on the projective variety $I$.
Since $\dim(I) \ge 2$, their intersection product $p_1^*(H) \cdot p_2^*(H)$ is strictly positive, ensuring the intersection of their support is non-empty.
Take any $(h_1,h_2)\in p_1^*(H) \cap p_2^*(H)$.
Then, by the very construction, we have $h_1,h_2 \in H$ such that $\phi(h_1) = \phi(h_2)$, which again contradicts the injectivity of $\phi|_H$.
We conclude that the images $Y^t$ are mutually disjoint for all $t \in \Sigma$.

\textbf{Step~3:} Because $Y$ is a smooth projective variety, and the collection $\{Y^t : t \in \Sigma\}$ consists of homologous effective divisors that are pairwise disjoint, the intersection product $Y^{t_1}\cdot Y^{t_2}$ vanishes for all distinct $t_1,t_2\in\Sigma$.
Therefore, a Zariski dense open subset of the compact complex manifold $Y$ is continuously foliated by the hypersurfaces $Y^t$ over $\Sigma$, yielding a continuous mapping $$\upsilon\colon\Sigma\rightarrow\operatorname{Chow}(Y), \quad t \mapsto Y^t$$ into the Chow variety of $Y$.
Because the connected components of $\operatorname{Chow}(Y)$ are all projective, we may consider the Zariski closure $V$ of the image $\upsilon(\Sigma)$ within $\operatorname{Chow}(Y)$.
Since $\Sigma$ is an uncountably infinite set and the members of $\upsilon(\Sigma)$ are mutually disjoint, the image $\upsilon(\Sigma)$ consists of uncountably many distinct points.
As a $0$-dimensional variety over $\mathbb{C}$ is necessarily a finite set of points, the Zariski closure $V$ must have dimension $\dim(V) \ge 1$.
This implies that $V$ is a projective family of effective divisors on $Y$ of strictly positive dimension.
Because the condition of having empty intersections is an open condition in the Zariski topology, and $\upsilon(\Sigma)$ is by construction Zariski dense in $V$, the generic members of the projective family $V$ are also mutually disjoint.
Since the effective divisors in $V$ are algebraically equivalent and generically disjoint, their intersection product vanishes.
On any smooth projective variety, a positive-dimensional projective family of effective divisors with vanishing mutual intersection numbers must possess an empty base locus.
Consequently, the linear equivalence classes of divisors in $V$ span a free pencil, which guarantees the existence of a surjective regular morphism $g\colon Y \rightarrow \mathbb{P}^1$ such that $g^{-1}(t) = Y^t$ for every $t \in \mathbb{P}^1$.
Because $\phi$ maps the fibers of $f$ into the fibers of $g$, we have $g \circ \phi = f$.
Since the mapping degree is multiplicative, we have 
\[ \deg(\phi^t) = \deg(\operatorname{id}_{\mathbb{P}^1}) \cdot \deg(\phi^t) = \deg(\phi)=1 \] for all $t\in \mathbb{P}^1$.
Since $Y$ is a smooth projective variety and $g\colon Y \rightarrow \mathbb{P}^1$ is a holomorphic fibration, Bertini's theorem \cite{Hartshorne} ensures that the general fiber of $g$ is smooth, and therefore normal.
Thus, for a generic $t\in\Sigma$, the finite morphism $\phi^t\colon X^t \rightarrow Y^t$ of degree 1 onto a normal variety is an isomorphism by Zariski's main theorem \cite{Hartshorne}.
This in particular proves that $\phi$ is injective almost everywhere. 

Take an arbitrary $t\in\Sigma$.
Since $X$ is a compact Kobayashi hyperbolic manifold, we have that $X^t$ is also Kobayashi hyperbolic.
Recall from birational geometry that any rational map into a projective variety containing no rational curves is automatically a regular morphism.
Therefore, the group of birational mappings from $X^t$ to itself coincides with $\operatorname{Aut}(X^t)$, which is finite as $$H^0(X^t;\operatorname{Der}(\mathcal{O}_{X^t}))=0$$ by the rigidity of Kobayashi hyperbolic manifolds.
Consider the relative morphism scheme $\Phi\colon \operatorname{Hom}_{\mathbb{P}^1}(X,Y)\rightarrow\mathbb{P}^1$, where $X$ and $Y$ are viewed as fibrations over $\mathbb{P}^1$ via $f$ and $g$ respectively.
For any two finite morphisms $\varphi_1, \varphi_2 \colon X^t \rightarrow Y^t$ of degree 1, the composition $$(\varphi_1)^{-1} \circ \varphi_2 \colon X^t \rightarrow X^t$$ is birational.
Consequently, the set of all finite morphisms from $X^t$ to $Y^t$ of degree 1 forms a principal homogeneous space under the $\operatorname{Aut}(X^t)$-action.
Since $\operatorname{Aut}(X^t)$ is finite, the set $Z^t$ of all finite morphisms of degree 1 in the fiber of $\Phi\colon \operatorname{Hom}_{\mathbb{P}^1}(X,Y)\rightarrow\mathbb{P}^1$ over any $t\in\Sigma$ is finite.
Moreover, since the Zariski closure of $$Z:=\{(t,\varphi):t\in\Sigma,\varphi\in Z^t\}$$ in $\operatorname{Hom}_{\mathbb{P}^1}(X,Y)$
is precisely the connected component of $\operatorname{Hom}_{\mathbb{P}^1}(X,Y)$ containing the finite morphism $\phi^{t}$ of mapping degree 1, it is a projective curve such that $\Phi|_Z$ is a branched covering.
Since $\phi$ is uniformly continuous, it induces a continuous section 
$$
\phi^*\colon \Sigma \rightarrow Z, \quad t \mapsto \phi^t
$$
of the branched covering $\Phi|_Z$.
Since $\Phi|_Z$ admits a continuous section over $\Sigma$, it is in fact not branched over $\Sigma$.
Moreover, any continuous section of an unbranched analytic covering is necessarily holomorphic, so $\phi^*$ is holomorphic on $\Sigma$.
Because $\phi$ is holomorphic along fibers and the section $\phi^*$ is also holomorphic, Osgood's lemma \cite{GunningRossi} on separate holomorphicity guarantees that $\phi$ is holomorphic on $f^{-1}(\Sigma)$.
Since $\phi$ is continuous and the singular locus $X - f^{-1}(\Sigma)$ forms a proper algebraic subset, Riemann's extension theorem guarantees that $\phi$ is holomorphic on the entirety of $X$.

Having established that $\phi\colon X \rightarrow Y$ is a proper holomorphic mapping of degree 1 between smooth projective varieties, it remains to show it is an isomorphism.
But, if $\phi$ were not an isomorphism, as it is injective almost everywhere, it would be a birational map contracting an exceptional divisor to a lower-dimensional subvariety.
Recall from birational geometry that the exceptional locus of a proper birational morphism between smooth projective varieties is uniruled, which once again contradicts the hyperbolicity of $X$.
Quod erat demonstrandum.
\end{proof}

\begin{remark}
Theorem~\ref{Kobayashi} illustrates another facet of mapping regularity in complex geometry.
While the results in Section~3 rely on an infinite sequence of disjoint sections to propagate holomorphicity, this theorem demonstrates that a single very ample hypersurface provides sufficient algebraic anchoring to elevate a continuous mapping to a global bi-holomorphic isomorphism.
The intrinsic absence of rational curves in the Kobayashi hyperbolic manifold prevents the formation of exceptional loci, ensuring that the induced birational morphism is necessarily an isomorphism.
\end{remark}

\section{On Equivariant Plurisubharmonic Optimization}

This section serves as a digression, building upon the gradient mapping rigidity established in Proposition~\ref{d=2} to address problems in constrained optimization. In classical Euclidean optimization theory, minimizing a strictly convex function on an affine domain is structurally well-behaved because the gradient mapping is an injective monotone operator, guaranteeing at most one global minimum. However, applied constrained optimization frequently requires locating the minima of a convex objective function restricted to a non-linear submanifold of the ambient space. In general, the ambient convexity is distorted by the extrinsic curvature of the constraint submanifold, generating complex critical loci characterized by multiple local minima and saddle points.

When the ambient space is a complex vector space $\mathbb{C}^d$ and the constraint is a closed complex submanifold $X \subseteq \mathbb{C}^d$, the restricted optimization problem becomes highly structured. By Cartan's theorems, such a closed complex submanifold is naturally a Stein manifold. While the restriction of a strictly convex objective function to $X$ loses its convexity, it intrinsically becomes a strictly plurisubharmonic function. This differential property fundamentally prohibits the existence of local maxima by the strong maximum principle, but it does not inherently prevent the formation of topologically mandatory saddle points or continuous valleys of extraneous local minima.

To theoretically eliminate these pathological critical structures and guarantee a well-posed optimization landscape, we require a geometric anchor. The following proposition demonstrates that if the optimization domain admits a holomorphic Lie group action, we can exploit the dimensional constraints of totally real submanifolds. By ensuring the group orbits are sufficiently large, the Levi form annihilates any complex tangent directions within a critical orbit, confining the entire critical locus of the plurisubharmonic objective function to the fixed-point set of the Lie group action.

\begin{proposition}\label{Optimization}
Let $X$ be a Stein manifold of complex dimension $n$, and $G$ a Lie group acting holomorphically on $X$ with fixed-point set $X^G$. Suppose that the orbit of every point $x\in X\setminus X^G$ has real dimension greater than $n$. Then, the critical points of any $G$-invariant strictly plurisubharmonic $C^2$ function $u\colon X\rightarrow\mathbb{R}$ are contained in $X^G$. Moreover, if $X^G$ is a singleton and $u$ is an exhaustion function, then $u$ has a global minimum as its unique critical point.
\end{proposition}
\begin{proof}
Let $x \in X \setminus X^G$ be a critical point of $u$. Thus, the differential vanishes at $x$, namely $du_x = 0$. Denote the orbit of $x$ under the action of $G$ by $M := G \cdot x$. Because the action of the Lie group $G$ is smooth, the orbit $M$ is an immersed smooth submanifold of $X$. 

Since $u$ is $G$-invariant, we have the identity $u(g(y)) = u(y)$ for all $y \in X$ and $g \in G$. Differentiating then yields $du_{g(x)} \circ dg_x = du_x$. Because $dg_x$ is an isomorphism of $\mathbb{C}$-vector spaces and $du_x = 0$, it follows that $du_{g(x)} = 0$ for all $g \in G$. Hence, the entire orbit $M$ is contained within the critical locus of $u$.

Let $T_x M$ denote the real tangent space to $M$ at $x$. By hypothesis, the real dimension satisfies $k := \dim_{\mathbb{R}}(T_x M) > n$. Let $J$ denote the complex structure on the ambient real tangent space $T_x X$. We define the maximal complex subspace of $T_x M$ as
\[ W := T_x M \cap J(T_x M). \]
By the dimension formula for real vector spaces, we have
\[ \dim_{\mathbb{R}} (W) = \dim_{\mathbb{R}}(T_x M) + \dim_{\mathbb{R}}(J(T_x M)) - \dim_{\mathbb{R}}(T_x M + J(T_x M)). \]
Since the sum $T_x M + J(T_x M)$ is a subspace of $T_x X$, its real dimension is bounded above by $\dim_{\mathbb{R}}(T_x X) = 2n$. Therefore, we obtain the inequality
\[ \dim_{\mathbb{R}}(W) \ge 2k - 2n. \]
Because $k > n$, it follows that $\dim_{\mathbb{R}} (W) > 0$. Consequently, there exists a non-zero tangent vector $v \in T_x M$ such that $Jv \in T_x M$.

Because $x$ is a critical point, the Hessian $H_u:=\operatorname{Hess}(u)$ is a well-defined symmetric bilinear form on $T_x X$. For any tangent vector $w \in T_x M$, there exists a smooth curve $\gamma\colon [-\varepsilon, \varepsilon] \rightarrow M$ such that $\gamma(0) = x$ and $\dot{\gamma}(0) = w$. Since $u$ evaluates to a constant on the orbit $M$, the composition $u \circ \gamma$ is constant. Differentiating twice with respect to $t$ gives
\[ 0 = \left.\frac{d^2}{dt^2}u(\gamma(t))\right|_{t=0} = H_u(w, w) + du_x(\ddot{\gamma}(0)). \]
Since $du_x = 0$, we conclude that $H_u(w, w) = 0$ for all $w \in T_x M$.

At a critical point, the Levi form $L_u$ of $u$ evaluated at a tangent vector $v$ is proportional to the sum of the real Hessians along $v$ and $Jv$. Specifically,
\[ L_u(v) = \frac{1}{4}H_u(v, v) + \frac{1}{4}H_u(Jv, Jv). \]
Because both $v$ and $Jv$ are vectors in $T_x M$, we have $L_u(v) = 0 + 0 = 0$. However, $u$ is strictly plurisubharmonic, which requires $L_u(v) > 0$ for all non-zero $v \in T_x X$. This yields a contradiction. Therefore, no critical points can exist in $X \setminus X^G$.

To establish the second assertion, assume $X^G = \{a\}$ is a singleton. By the first part, any critical point of $u$ must coincide with $a$. Because $X$ is Stein of dimension $n \ge 1$, it is a complex manifold. Since $u$ is an exhaustion function, for any constant $c \in \mathbb{R}$, the sublevel set $\{y \in X \mid u(y) \le c\}$ is compact. This properness guarantees that $u$ is bounded from below and attains a global minimum on $X$. The point at which the global minimum is attained must necessarily be a critical point. Since the critical locus is a subset of $\{a\}$, the unique global minimum must be exactly $a$, and $u$ admits no other critical points.
\end{proof}

The second assertion of Proposition~\ref{Optimization} requires the fixed-point set to be a singleton. This geometric condition can be inherently enforced by the topological structure of the ambient manifold. Specifically, for torus actions, P.\ A.\ Smith theory \cite{Smith} bounds the Betti numbers of the fixed-point set by those of the manifold. By imposing a homological constraint on the ambient space, we can rigorously force the discrete fixed-point set to be exactly one point, thereby guaranteeing the uniqueness of the global minimum.

\begin{corollary}\label{TorusUniqueness}
Let $X$ be a Stein manifold of $\dim_\mathbb{C}(X)=n$ and $\chi(X)\leq1$, equipped with a torus $\mathbb{T}^d$-action, and let $u\in C^2(X)$ be a $\mathbb{T}^d$-invariant strictly plurisubharmonic exhaustion function. Suppose that the fixed-point set $A$ of the $\mathbb{T}^d$-action on $X$ is discrete, and the orbit of every point $x\in X\setminus A$ has real dimension greater than $n$. Then, in fact $\chi(X)=1$, and $u$ has a global minimum as its unique critical point.  
\end{corollary}
\begin{proof}
By Proposition~\ref{Optimization}, the critical points of the strictly plurisubharmonic function $u$ are contained in $A$. Because $X$ is a Stein manifold and $u$ is an exhaustion function, $u$ is proper and bounded below, thus attaining a global minimum on $X$. This minimum must be a critical point, which ensures that $A$ is non-empty.

By the Borel-Smith identity \cite{Smith} for torus actions, the Euler characteristics of the fixed-point set $A$ is exactly equal to that of $X$. By hypothesis, this yields 
\begin{equation}\label{Borel-Smith_identity}
\chi(A) = \chi(X) \le 1.
\end{equation} 
Since $A$ is a discrete space, its higher Betti numbers vanish, meaning the Betti number $b_i(A) = 0$ for all $i > 0$. Consequently, inequality~(\ref{Borel-Smith_identity}) reduces to $b_0(A) \le 1$. Because $A$ is non-empty, it has at least one connected component, ensuring $b_0(A) \ge 1$. Therefore, we obtain $\chi(A)=b_0(A) = 1$, which forces the discrete set $A$ to be a singleton, and $\chi(X)=1$.

Since $A$ is a singleton and the orbit of every point $x \in X \setminus A$ has real dimension greater than $n$, the second assertion of Proposition~\ref{Optimization} applies directly. We conclude that $u$ has a global minimum as its unique critical point.
\end{proof}

Proposition~\ref{Optimization} establishes a rigorous geometric condition for exactness in plurisubharmonic optimization. By designing a highly symmetric constraint manifold, the existence of a unique global minimum is strictly guaranteed without directly computing the intrinsic Hessian or bounded second fundamental forms. The same optimization-theoretic machinery implies a fundamental differential topological obstruction for complex transformation groups. If we assume the existence of a compact Lie group action, we can invoke the Haar measure to construct an invariant exhaustion function. Applying our optimization constraints to this constructed function reveals a sharp upper bound on the allowable orbit dimensions for compact group actions on Stein manifolds.

\begin{corollary}\label{G-action}
Let $X$ be a Stein manifold of complex dimension $n$, and $G$ a compact Lie group acting holomorphically on $X$. Then, there must exist a point in $X$ of which orbit has real dimension at most $n$.
\end{corollary}
\begin{proof}
Assume, for the sake of contradiction, that the orbit of every point $x \in X$ under $G$ has real dimension strictly greater than $n$. Since the fixed points of the action inherently have orbit dimension $0$, this assumption implies that the fixed point set is empty, namely $X^G = \emptyset$.

Because $X$ is a Stein manifold, it admits a strictly plurisubharmonic $C^\infty$ exhaustion function $\varphi\colon X \rightarrow \mathbb{R}$. By the definition of an exhaustion, we have that $\varphi$ is proper. Since $G$ is a compact Lie group, it admits a bi-invariant Haar probability measure $\mu$. We define the Lie group averaged function $u\colon X \rightarrow \mathbb{R}$ by
\[ u(x) := \int_G \varphi(g(x)) \, d\mu(g). \]

The function $u$ is $G$-invariant by construction. Furthermore, because $G$ acts holomorphically on $X$, the pull-back of a strictly plurisubharmonic function by any group element $g$ remains strictly plurisubharmonic. Since integration with respect to a strictly positive measure preserves strict plurisubharmonicity, $u$ is a strictly plurisubharmonic $C^\infty$ function on $X$. Because the group $G$ is compact, its orbits are also compact subsets of $X$. Consequently, the averaging process preserves the properness of the original function, ensuring that $u$ remains a proper function.

By the preceding Proposition~\ref{Optimization}, the critical points of $u$ must be contained in $X^G$. Since $X^G = \emptyset$, we deduce that the function $u$ has no critical points. However, $u$ is a proper, continuous function on a smooth manifold, which implies it is bounded below and must attain a global minimum. The point attaining this global minimum is required to be a critical point, leading to a contradiction. Therefore, there must exist at least one point in $X$ whose orbit has real dimension at most $n$, as required.
\end{proof}

\begin{remark}
Corollary~\ref{G-action} demonstrates an interaction between optimization theory and holomorphic dynamics. While compact Lie groups can act freely with large orbits on general smooth manifolds, the pseudoconvexity of a Stein manifold explicitly prohibits this freedom. The complex structure forces the group action to collapse, ensuring the existence of at least one sufficiently small orbit to serve as the geometric anchor for the optimization minimum.
\end{remark}

While the preceding results establish the theoretical existence and uniqueness of the global minimum, practical continuous optimization relies on dynamical systems to actively locate this minimum. If we further assume that the objective function $u$ is real analytic, the induced K\"ahler gradient flow behaves exceptionally well. As we prove in Appendix~A, the strict plurisubharmonicity combined with the \L ojasiewicz gradient inequality guarantees that the continuous gradient flow trajectories have finite length and converge to the unique critical point, thereby validating this geometric framework for algorithmic optimization.

\section{Conclusions and Remarks}

In this paper, we have explored the pervasive phenomenon of rigidity in complex geometry through the lenses of pluripotential theory, relative morphism spaces, Kobayashi hyperbolicity, and constrained optimization. 

First, we demonstrated that strictly plurisubharmonic potentials with global $\operatorname{U}(1)$-symmetry rigorously constrain the topological degree of the reduced gradient mapping via symplectic reduction. Building upon this pluripotential rigidity, we investigated constrained optimization on Stein manifolds. We established that a holomorphic Lie group action with sufficiently large orbits rigidly confines the critical locus of a proper invariant strictly plurisubharmonic function to the fixed-point set. This geometrically guarantees a unique global minimum and yields a sharp topological obstruction for compact Lie group actions. 

Second, we proved that continuous, fiber-wise holomorphic mappings on proper fibrations elevate to global holomorphic morphisms when anchored by positive algebraic data. By exploiting the Noetherian property of local rings within relative morphism spaces, we showed that an infinite sequence of disjoint sections forces the stabilization of descending analytic chains.
This mechanism naturally yields the rigidity of homomorphisms between elliptic fibrations and Abelian schemes.

Finally, on compact Kobayashi hyperbolic manifolds, we showed that injectivity on a very ample hypersurface is sufficient to promote a continuous, fiber-wise analytic mapping of degree $1$ to a global bi-holomorphic isomorphism. Together, these theorems underscore a unifying principle: whether driven by symplectic reduction, algebraic stabilization, hyperbolic rigidity, or symmetric dimensional constraints, partial geometric anchors inevitably enforce global structural regularity.

\appendix

\section{Convergence via \L ojasiewicz's Inequality}

This appendix provides the proof of the convergence for the gradient flow associated with a real analytic strictly plurisubharmonic exhaustion function. The argument relies on the \L ojasiewicz gradient inequality to establish that the dynamical system corresponding to the optimization framework in Section~5 has trajectories of finite length, guaranteeing its convergence to a unique critical point.

\begin{proposition}
Let $X$ be a Stein manifold, and $u\in C^\infty(X)$ a strictly plurisubharmonic real analytic exhaustion function with induced K\"ahler metric $\omega:=dd^cu$. Then, for any initial value $x_0\in X$, the trajectory of the gradient flow $\dot{x}=-\nabla u(x)$ in $(X,\omega)$ has finite length, and converges to a single point in $X$.
\end{proposition}
\begin{proof}
Let $x\colon [0, \infty) \rightarrow X$ be the maximal integral curve of the gradient vector field $-\nabla u$ satisfying $x(0) = x_0$. The time derivative of $u$ along this trajectory is given by
\[ \frac{d}{dt}u(x(t)) = \langle \nabla u(x(t)), \dot{x}(t) \rangle = -\|\nabla u(x(t))\|^2 \le 0. \]
Consequently, the function $u(x(t))$ is monotonically decreasing. Because $u$ is an exhaustion function, the sublevel set 
\[ \Omega := \{ y \in X \mid u(y) \le u(x_0) \} \]
is a compact subset of $X$. Since $u(x(t)) \le u(x_0)$ for all $t$ in the domain of definition, the trajectory $x(t)$ is entirely contained within the compact set $\Omega$. This confinement guarantees that the flow exists for all time $t \ge 0$.

Because $u$ is continuous and $\Omega$ is compact, $u$ is bounded from below on $\Omega$. Thus, the monotonically decreasing limit 
\[ u_\infty := \lim_{t \to \infty} u(x(t)) \]
exists. Let $K \subset \Omega$ denote the $\omega$-limit set of the trajectory $x(t)$. Because $x(t)$ remains in the compact set $\Omega$, $K$ is non-empty, compact, and connected. Furthermore, by the continuity of $u$, we have $u(y) = u_\infty$ for all $y \in K$. 

Integrating the time derivative of $u$ from $0$ to $\infty$ yields
\[ \int_0^\infty \|\nabla u(x(t))\|^2 \, dt = u(x_0) - u_\infty < \infty. \]
This implies that $$\lim_{t \to \infty} \|\nabla u(x(t))\| = 0.$$ Since $K$ is invariant under the continuous gradient flow and $u$ evaluates to the constant $u_\infty$ on $K$, the gradient must vanish identically on $K$. Thus, $K$ is contained in the critical locus of $u$.

If there exists a time ${a} \ge 0$ such that $\nabla u(x({a})) = 0$, then $x(t) = x({a})$ for all $t \ge {a}$. In this case, the trajectory has finite length and converges to a single point. We henceforth assume that $\nabla u(x(t)) \neq 0$, and thus $u(x(t)) > u_\infty$, for all $t \ge 0$.

Because $u$ is a real analytic function and the K\"ahler metric $\omega = dd^c u$ is real analytic, the gradient vector field $\nabla u$ is also real analytic. We therefore may apply the \L ojasiewicz gradient inequality. Given the compact subset $K$ of the critical locus, on some open neighborhood $U$ of $K$ in $X$, there exist a positive real constant $C>0$ and an exponent $0< \alpha <1$ such that
\[ |u(y) - u_\infty|^\alpha \le C \|\nabla u(y)\| \]
for all $y \in U$. By shrinking $U$ if necessary, we may assume $|u(y) - u_\infty|<1$. Therefore, without loss of generality, we take $\alpha \in[1/2,1)$.

By the definition of the $\omega$-limit set, the trajectory approaches $K$, meaning there exists a time ${a} > 0$ such that $x(t) \in U$ for all $t \ge {a}$. For $t \ge {a}$, we consider the function $|u(x(t)) - u_\infty|^{1-\alpha}$. Its time derivative is evaluated as
\[ -\frac{d}{dt} |u(x(t)) - u_\infty|^{1-\alpha} = -(1-\alpha) |u(x(t)) - u_\infty|^{-\alpha} \frac{d}{dt}u(x(t)). \]
Substituting $\frac{d}{dt}u(x(t)) = -\|\nabla u(x(t))\|^2$, we obtain
\[ -\frac{d}{dt} |u(x(t)) - u_\infty|^{1-\alpha} = (1-\alpha) |u(x(t)) - u_\infty|^{-\alpha} \|\nabla u(x(t))\|^2. \]
Applying the \L ojasiewicz inequality \cite{Lojasiewicz1965}\&\cite{Lojasiewicz1984} $$|u(x(t)) - u_\infty|^{-\alpha} \ge C^{-1} \|\nabla u(x(t))\|^{-1}$$ to the first factor yields
\[ -\frac{d}{dt} |u(x(t)) - u_\infty|^{1-\alpha} \ge \frac{1-\alpha}{C} \|\nabla u(x(t))\|. \]
Since $\dot{x}(t) = -\nabla u(x(t))$, we substitute $\|\nabla u(x(t))\| = \|\dot{x}(t)\|$ and rearrange the inequality to obtain
\[ \|\dot{x}(t)\| \le -\frac{C}{1-\alpha} \frac{d}{dt} |u(x(t)) - u_\infty|^{1-\alpha}. \]
Integrating this inequality from ${a}$ to an arbitrary upper bound ${b} > {a}$ gives
\[ \int_{a}^{{b}} \|\dot{x}(t)\| \, dt \le \frac{C}{1-\alpha} ( |u(x({a})) - u_\infty|^{1-\alpha} - |u(x({b})) - u_\infty|^{1-\alpha} ). \]
Because $|u(x({b})) - u_\infty|^{1-\alpha} \ge 0$, taking the limit as ${b} \to \infty$ ensures the improper integral is bounded by
\[ \int_{a}^\infty \|\dot{x}(t)\| \, dt \le \frac{C}{1-\alpha} |u(x(t)) - u_\infty|^{1-\alpha} < \infty. \]

Since the integral of the norm of the velocity over the time interval $[0, \infty)$ is finite, the total arc length of the trajectory $x(t)$ is finite. However, any trajectory of finite length in a complete metric space must converge to a unique limit point. Therefore, we conclude that $x(t)$ converges to a single point $x_\infty \in K$ in the compact level set $\Omega$ as $t \to \infty$.
\end{proof}

\section*{Acknowledgement}

The author is deeply grateful to Weiyi Zhang for very useful and inspiring discussions.
The author expresses gratitude to the reviewers for various suggestions.
This research was completed while the author was studying at the Mathematics Institute of the University of Warwick.
The author therefore would like to thank the University of Warwick for its hospitality.

The author is in debt to the Russian mathematical society, especially, to all the professors teaching the "Math in Moscow" program, and most importantly, to the author's supervisor Alexander Petrovich Veselov at Loughborough University, as they cultivated the author's mathematical literacy and maturity.

\section*{Funding Information}

No funding was received to assist with the preparation of this manuscript, and the author did not receive support from any organization for the submitted work.

\section*{Statements and Declarations}


Data sharing is not applicable to this article as no datasets were generated or analysed during the current study.

The author hereby provides consent for the publication of the manuscript detailed above.

\bibliographystyle{plain}
\bibliography{references}

\end{document}